\documentclass[a4paper,10pt,reqno, english]{amsart}  
\usepackage[utf8]{inputenc}
\usepackage[T1]{fontenc}
\usepackage{amsmath,amsthm}
\usepackage{amsfonts,amssymb,enumerate}
\usepackage{url}
\usepackage{paralist}
\usepackage[colorlinks=true,urlcolor=blue,linkcolor=red,citecolor=magenta]{hyperref}
\usepackage{enumerate}
\usepackage{anysize}
\usepackage{tikz, float, mathtools}

\theoremstyle{plain}
\newtheorem{thm}{Theorem}

\newtheorem*{thm*}{Theorem}

\theoremstyle{definition}

\newcommand{\R}{\mathbb{R}}

\DeclareMathOperator{\conv}{conv}

\begin{document}

\title{Achieving rental harmony with a secretive roommate}



\author{Florian Frick \and Kelsey Houston-Edwards \and Fr\'ed\'eric Meunier}

\address[FF, KHE]{Department of Mathematics, Cornell University, Ithaca, NY 14853, USA}
\email{\{ff238,kah282\}@cornell.edu}

\address[FM]{Universit\'e Paris Est, CERMICS, 77455 Marne-la-Vall\'ee CEDEX, France}
\email{frederic.meunier@enpc.fr}

\date{\today}
\maketitle


\begin{abstract}
\small
Given the subjective preferences of $n$ roommates in an $n$-bedroom apartment, one can use Sperner's lemma to find a division of the rent such that each roommate is content with a distinct room. At the given price distribution, no roommate has a strictly stronger preference for a different room. We give a new elementary proof that the subjective preferences of only $n-1$ of the roommates actually suffice to achieve this envy-free rent division. Our proof, in particular, yields an algorithm to find such a fair division of rent. The techniques also give generalizations of Sperner's lemma including a new proof of a conjecture of the third author.
\end{abstract}

\section{Introduction}

As the rent in Larry's new two-bedroom apartment is \$1000 he will have to look for a roommate to split the cost.
The two rooms are not the same size, and each has their own advantages, yet fairly dividing the rent between
the rooms is Larry's primary concern. Moe is interested in moving in with Larry, and Larry feels that splitting the rent
\$600 -- \$400 between the two rooms is fair --- for this division of the rent the disadvantages of the second room
are offset by its reduced cost. Now, when Larry offers the two rooms to Moe at these prices, it will not matter to
him which room Moe chooses; Larry will be content with the other room. The two new roommates will not be envious
of one another and live in a state of \emph{rental harmony}. Larry accomplished this envy-free rent division without 
taking Moe's preferences into account.

This is not a lucky accident of the two person -- two bedroom situation: for a three-bedroom apartment Larry and Moe
can fairly divide the rent among the rooms without taking the preferences of a third roommate, Curly, into account. There is
a division of the total rent such that Curly can decide on an arbitrary room, and this will leave Larry and Moe with
sufficiently many options to accomplish rental harmony among the three of them. In general, it suffices if $n-1$ roommates 
know each others preferences to fairly divide the rent of an $n$-bedroom apartment. We give an algorithm for producing such a fair division of rent; see Asada et. al.~\cite{asada2017} for the recent 
nonconstructive topological proof of this result.

That rental harmony can always be achieved in an $n$-bedroom apartment (under mild conditions) if the subjective 
preferences of all $n$ future roommates are known was shown by Su~\cite{su1999}, partially reporting on work of Simmons.
The proof uses a combinatorial-geometric lemma about labelings of simplices due to Sperner~\cite{sperner1928}. This makes the proof,
especially for low~$n$, accessible to a nonexpert audience. Here our goal is to adapt Su's arguments for $n=3$ and then give a separate elementary proof of the existence of a fair division of rent for $n$ roommates, where the preferences of one roommate are unknown.

We first recall the mild conditions stipulated by Su to guarantee the existence of an envy-free rent division:
\begin{compactenum}[1.]
	\item In any division of the rent, each tenant finds at least one room acceptable.
	\item Each person prefers a room which costs no rent (i.e., a free room) to a non-free room.
	\item If a person prefers a room for a convergent sequence of prices, then that person also prefers the room for the limiting price.
\end{compactenum}

\noindent
Here we will slightly change condition 3 to simplify the proof and eliminate the need to take limits: 
The roommates do not care about one-cent error margins. We remark that Su's second condition
together with the third condition imply that the roommates are indifferent among free rooms, that is, if multiple rooms are free then
each roommate is content with any of them. This is because for any division of the rent where multiple
rooms are free, there is always a sequence of prices converging to this rent division where only one
specific room is free. And thus, we will also assume that the roommates are indifferent among free rooms.
In summary, our conditions are:
\begin{compactenum}[1.]
	\item In any division of the rent, each tenant finds at least one room acceptable.
	\item Each person prefers a room which costs no rent (i.e., a free room) to a non-free room, and each person is indifferent among free rooms.
	\item The roommates do not care about one-cent error margins.
\end{compactenum}

\medskip
\noindent 
Under these conditions, we can give a new, elementary (and now constructive) proof of our main theorem:

\begin{thm}
\label{thm:secret}
	For an $n$-bedroom apartment it is sufficient to know the subjective preferences of $n-1$ roommates to
	find an envy-free division of rent.
\end{thm}

In Section~\ref{sec:sperner} we recall Sperner's lemma and provide two proofs. The first proof is based on a classical ``trap-door'' argument and the second proof introduces a piecewise linear map which is used in the proof of our main theorem. Section~\ref{sec:proof3case} gives two proofs of Theorem~\ref{thm:secret} in the case where~${n=3}$, the first of which was originally presented by the second author in the \emph{PBS Infinite Series} episode ``Splitting Rent with Triangles''~\cite{houston-edwards2017}. Section~\ref{sec:general} generalizes the second proof given in Section~\ref{sec:proof3case} to show the main theorem. Section~\ref{sec:algorithm} explains how this proof yields an algorithm to find the fair division of rent. Section~\ref{sec:generalizations} utilizes the piecewise linear map introduced in previous sections to prove two generalizations of Sperner's lemma in an elementary way --- one of these generalizations had been conjectured by the third author and was recently proven by Babson~\cite{babson2012} with different methods.

\section{Sperner's lemma}
\label{sec:sperner}

Begin with a triangle which is subdivided into several smaller triangles. Label the three vertices of the original triangle from the 
set~$\{1,2,3\}$ so each vertex receives a distinct label. Then, label each vertex on the edges of the original triangle by either of 
the labels at the endpoints. Finally, label the interior vertices $1$, $2$, or~$3$ arbitrarily. This is a \emph{Sperner labeling} of 
a triangle, as in Figure~\ref{spernerlabeling}.

The result known as \emph{Sperner's lemma} for a triangle states that there exists some small triangle that exhibits each of the three labels on its vertices. Even stronger, there will be an odd number of such fully labeled triangles. Here we outline two proofs of Sperner's lemma. The first, known as the \emph{trap-door argument} goes back to Cohen~\cite{cohen1967} and Kuhn~\cite{kuhn1968}. The second is based on a piecewise linear map between the vertex labels. We use this map in Section~\ref{sec:proof3case} to prove there exists a fair division of rent for three roommates, one of whose preferences are secret.

\begin{figure}
	\centering
	\includegraphics[width=0.8\textwidth]{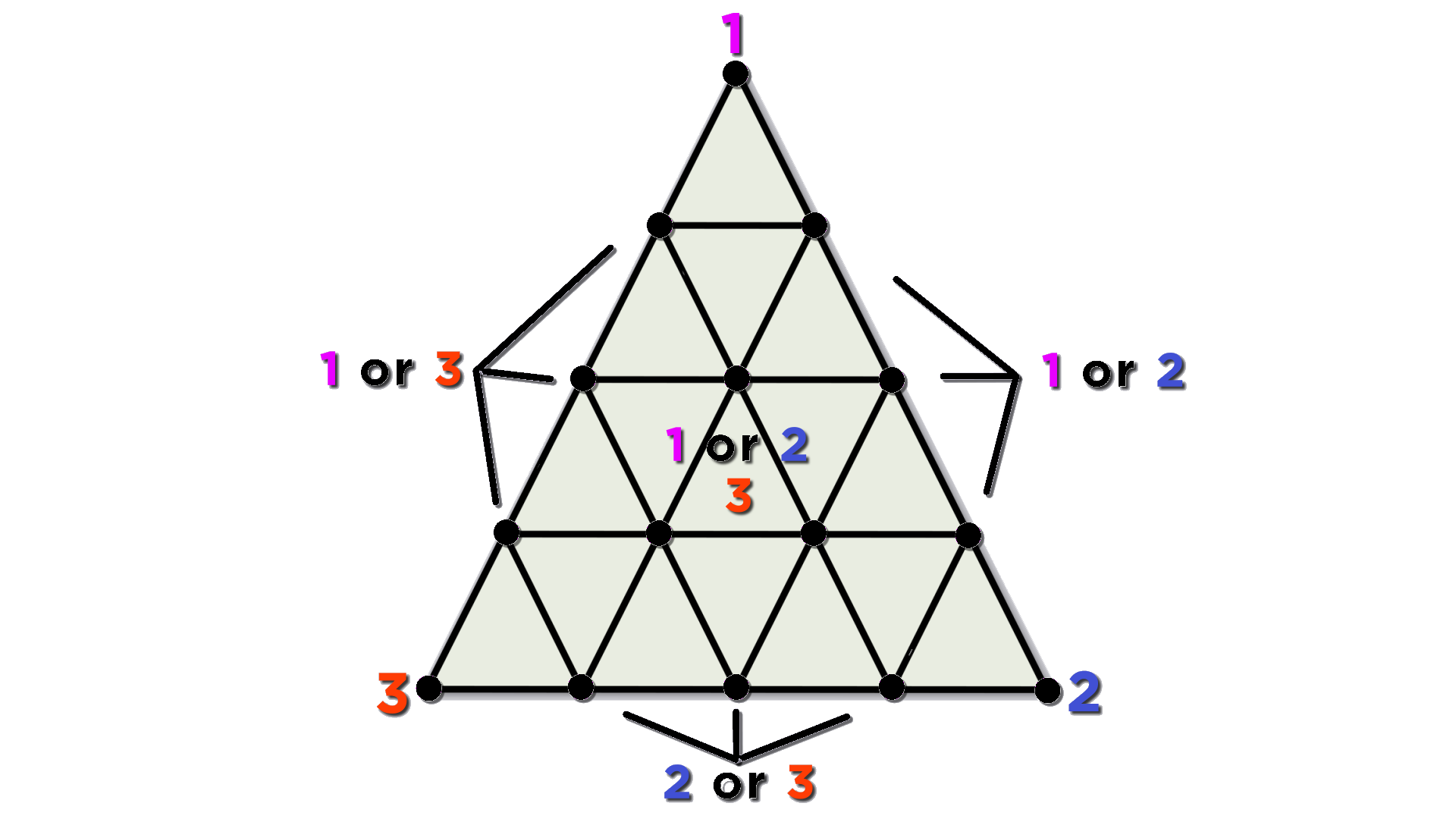}
	\caption{Options for a Sperner labeling}
	\label{spernerlabeling}
\end{figure}

\begin{proof}[Proof 1, Trap Door Argument:]
	Select two distinct vertices of the main triangle. To illustrate, we will use those labeled~$1$ and~$2$. Imagine each small triangle is a room and 
	any edge with endpoints labeled~$1$ and~$2$ is a door. Notice that each room has zero, one or two doors; it is impossible to have three doors. 
	The crucial observation is that a room has one door if and only if it is labeled with all three labels. To prove Sperner's lemma, we will find a room 
	with exactly one door. 
	
	Observe that there must be an odd number of doors along the boundary of the original triangle. All the boundary doors will fall on one edge of the original 
	triangle, between the main vertices labeled~$1$ and~$2$. Starting from the vertex labeled~$1$ and going down to the vertex labeled~$2$ at the bottom, 
	each time we encounter a smaller edge where the label switches from~$1$ to~$2$ or from~$2$ to~$1$, it's a door. It must ``switch'' an odd number of 
	times, so there are an odd number of exterior doors.
	
	Now, pick one of these exterior doors and walk through it. Either that was the only door in the room, or there is exactly one other door. If there is another 
	door, walk through that. Keep walking through doors -- without backtracking -- until you are stuck. This procedure leads to a room with only door, i.e., a 
	fully labeled room. See Figure~\ref{labelwithdoors}.
	
	However, following this procedure, it is possible to wander back out through a door on the boundary of the main triangle. This is why it is important that 
	there are an odd number of doors on the boundary. Entering through one door and exiting through another still leaves an odd number of boundary 
	doors -- namely, at least one -- to continue this procedure.

	Because there is an odd number of boundary doors, this procedure will result in an odd number of fully labeled rooms. However, there may be other 
	fully labeled rooms which are not accessible from the boundary. This will happen if and only if the single door of the fully labeled room leads to another 
	inaccessible fully labeled room. Since these inaccessible fully labeled rooms always come in pairs, the total number of fully labeled rooms must be odd.
\end{proof} 

The second proof will use the fact that a piecewise linear self map $\lambda\colon\Delta\longrightarrow \Delta$ of the triangle $\Delta$ that maps the boundary
to itself and has degree one on the boundary of the triangle (i.e., going around the boundary of $\Delta$ the image points wander once around the boundary as
well) must be surjective. While this is still intuitive, we will use a similar statement in higher dimensions when generalizing this proof: a piecewise 
linear self map $\lambda$ of the $n$-simplex $\Delta_n$ that preserves faces setwise (i.e., $\lambda(\sigma) \subseteq \sigma$ for any face $\sigma$ of $\Delta_n$) 
has degree one on the boundary and thus is surjective. We will give a simple path-following proof of the fact that any such map is surjective in Section~\ref{sec:algorithm}. In particular, any topological fact used in our argument will be 
proved in an elementary fashion.

	\begin{figure}
		\centering
		\includegraphics[width=0.5\textwidth]{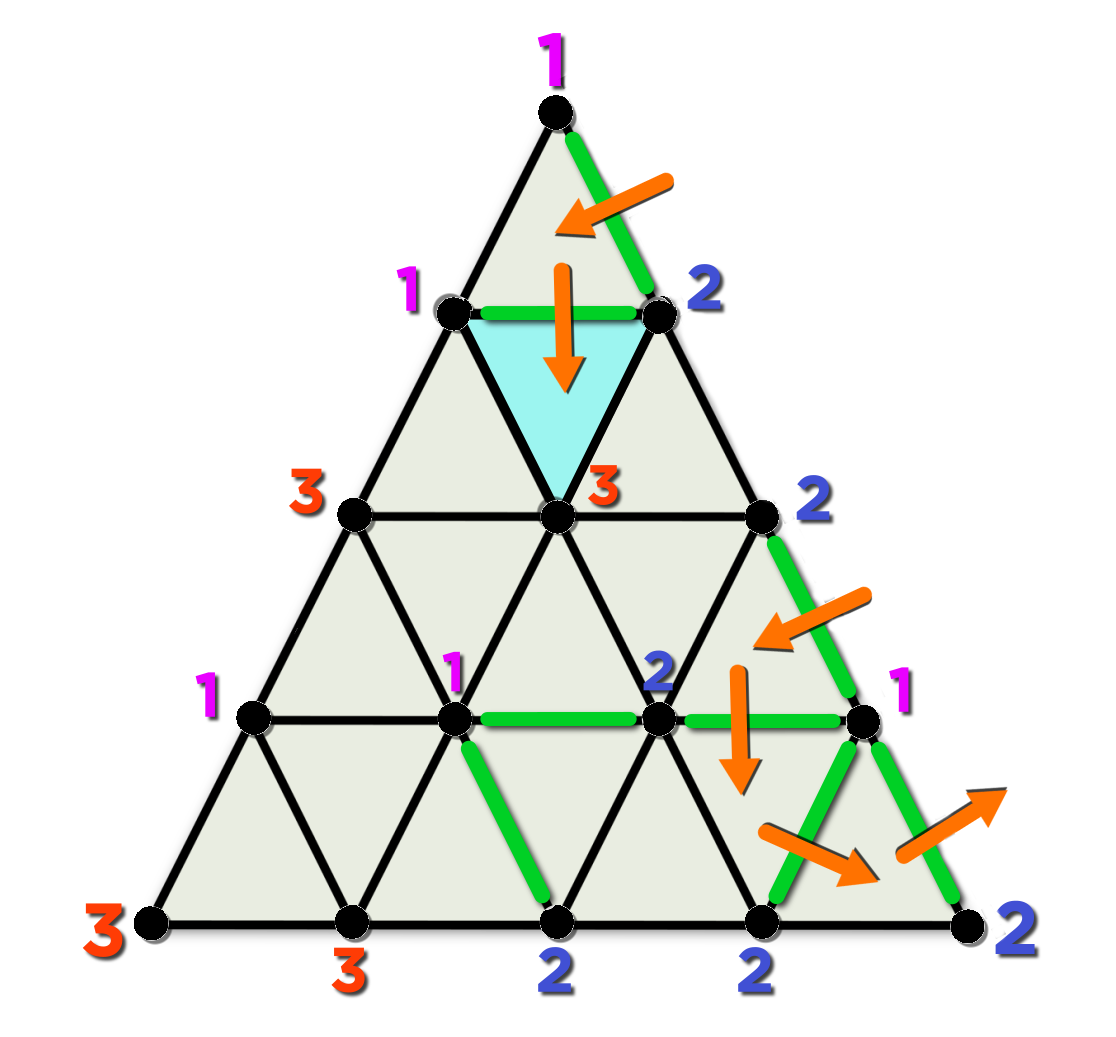}
		\caption{Procedure for finding fully labeled room}
		\label{labelwithdoors}
	\end{figure}

\begin{proof}[Proof 2, Piecewise Linear Map:]
	A Sperner labeling of a subdivided triangle can naturally be thought of as a piecewise linear map $\lambda \colon \Delta \longrightarrow \Delta$ from the triangle to itself, defined as follows. For each small triangle, $\lambda$ maps each vertex to the vertex of the original triangle with the same label. Then extend this map linearly within the small triangle. For example, the barycenter of a small triangle labeled $1,2,3$ maps to the barycenter of the original triangle,
	while the barycenter of a triangle labeled $1,1,2$ maps to the point on the edge with endpoints labeled $1$ and $2$ that separates the edge in a two-to-one ratio. Note that, for any edge $e$ of the original triangle $\lambda(e) \subseteq e$ (thus $\lambda$ also fixes the vertices), and that $\lambda$ is continuous. 
	
	Moving around the boundary of the original triangle and tracing the image points of~$\lambda$, we can see that $\lambda$ has degree one on the boundary of the original triangle. This implies that it is surjective and, in particular, there exists a point $x \in \Delta$ such that $\lambda(x)$ is the barycenter of the original triangle, i.e., $\lambda(x) = (\frac13, \frac13, \frac13)$. Then $\tau$, the smaller triangle containing $x$, must be fully labeled. If $\tau$ had only two of the labels, then $\lambda(\tau)$ would be contained in the edge of the original triangle spanned by the vertices with the two labels of $\tau$. This contradicts that $\lambda(x)$ is the barycenter.
\end{proof}

Both proofs of Sperner's lemma easily generalize to higher dimensions. Given an $(n-1)$-dimensional simplex~$\Delta_{n-1}$ that is subdivided into smaller $(n-1)$-dimensional simplices, a \emph{Sperner labeling} is a labeling of the vertices with $\{1,2,...,n\}$ such that (1) the $n$ vertices of the original $(n-1)$-simplex receive distinct labels, and (2) a vertex subdividing a $k$-face of~$\Delta_{n-1}$ is labeled by one of the $k+1$ labels of the $k$-face. \emph{Sperner's lemma} for higher-dimensional simplices states that any subdivided $(n-1)$-simplex with a Sperner labeling contains an odd number of smaller $(n-1)$-simplices that exhibit all $n$ labels on their vertices. More precisely, by a subdivision of $\Delta_{n-1}$ into smaller $(n-1)$-simplices, we mean that $\Delta_{n-1}$ is covered by a collection of $(n-1)$-simplices whose interiors are disjoint and such that any two of them intersect in a (possibly empty) face of both of them. We call such a subdivision a \emph{triangulation} of~$\Delta_{n-1}$.

The first proof of Sperner's lemma, using the trap-door argument, generalizes by induction. Assuming the lemma holds up to dimensions~${n-1}$, one can prove it holds for $n$-dimensional simplices by treating the $(n-1)$-dimensional faces labeled $(1,2,...,n)$ as doors. The proof of the existence of a fully labeled $n$-simplex then follows by direct analogy with the $n=2$ case above. A fully labeled $(k-1)$-dimensional face is a boundary door for a $k$-dimensional face, so one can link together the paths created in the trap-door argument to obtain the fully labeled room.

The second proof generalizes to higher-dimensional simplices more directly. The map $\lambda\colon \Delta_n \longrightarrow \Delta_n$ from the $n$-simplex to itself is defined nearly identically: Map each vertex of a smaller $n$-simplex to the vertex of the original simplex with the same label and then extend the map linearly inside each smaller simplex. 
Again, $\lambda$ has degree one on the boundary: the rules of a Sperner labeling impose that for any face $\sigma$ of the simplex $\Delta_n$ the image $\lambda(\sigma)$ is contained in~$\sigma$, since subdivision vertices contained in $\sigma$ are only labeled by labels found at the vertices of~$\sigma$. 
Thus $\lambda|_{\partial\Delta_n}$ is homotopic to the identity by induction on skeleta, which implies it has degree one. One can build the homotopy dimension by dimension using the fact that homotopies extend from lower-dimensional skeleta of a simplicial complex.
There must exist a point $x \in \Delta_n$ such that $\lambda(x) = (\frac{1}{n+1},\frac{1}{n+1},...,\frac{1}{n+1})$ is the barycenter. (As mentioned earlier, we will give an elementary proof of this fact in Section~\ref{sec:algorithm}.) The smaller $n$-simplex containing $x$ must be fully labeled.

\section{Rental harmony in the absence of full information for three roommates}
\label{sec:proof3case}

Here we give two combinatorial proofs of  the $n=3$ case of Theorem~\ref{thm:secret}, which mirror the two proofs of Sperner's lemma in Section~\ref{sec:sperner}. 
The first proof is constructive and yields an algorithm to find the fair division of rent, since it reduces Theorem~\ref{thm:secret} to Sperner's lemma for $n=3$. As presented in this section, the second proof only gives the existence of a fair division with a secret preference, but Section~\ref{sec:algorithm} explains how this method also yields an algorithm.

\begin{proof}[Proof 1:]	
	
	Form a triangle in $\R^3$ whose vertices are given by $(1,0,0), (0,1,0)$, and $(0,0,1)$. Because this lies in the plane $x+y+z=1$, we can interpret each point in the triangle as a division of rent. For example, $(\frac{1}{4},\frac{1}{4},\frac{1}{2})$ indicates that rooms 1 and 2 cost one-quarter of the total rent and room 3 costs one-half of the total rent. 
	
	Subdivide this triangle into many smaller triangles, i.e., triangulate the original triangle. If this triangulation is fine enough, then the vertices of a single small triangle represent divisions of the rent that only differ from one another by a penny or so. 
	At each vertex we survey both Larry and Moe, asking which room they would prefer if the rent were split in this specific way. We record their preferences as a tuple of two integers $(a,b)$ with $1 \le a,b \le 3$, where $a$ is the number of Larry's preferred room and $b$ is the number of Moe's preferred room. There are three vertices of the triangulation -- the original
	vertices of the triangle -- where two rooms are free. For those rent divisions we kindly ask Larry and Moe to make distinct choices,
	that is, the original vertices of the triangle all have labels $(a,b)$ with $a \ne b$. 
	
	At every vertex of the triangulation transform the label $(a,b)$ into a single-digit label according to the following rules: label the vertex~$3$
	if the previous label was $(1,1)$, $(1,2)$, or $(2,1)$, label~$1$ for $(2,2)$, $(2,3)$, or $(3,2)$, and $2$ for $(3,3)$, $(3,1)$, or $(1,3)$. The
	three original vertices of the triangle receive pairwise distinct labels since Larry and Moe decided for different rooms whenever
	two rooms were free simultaneously. On every edge of the triangle the vertices of the triangulation have the same label --- only
	one room is free. And this label is one of the two labels we find at the endpoints of that edge. Therefore we have defined a Sperner labeling.
	Thus we can find a small triangle that has vertices with all three labels. The corresponding rent divisions are all within a small margin of
	error. It is simple to check that seeing all three labels for this rent division (with small error margin) implies that Larry and Moe each prefer at least two distinct rooms, and each room is preferred by at least one of them. For example, it is impossible that Larry only likes room~$1$ since one vertex is labeled~$1$; and it is impossible that both Larry and Moe dislike room~$3$
	since one vertex is labeled~$2$. This means that regardless of which room Curly chooses, both Larry 
	and Moe will be left with one of their favorite rooms.
\end{proof}

The reduction in the proof above actually proves a generalization of Sperner's lemma: given two Sperner labelings of a subdivided triangle that
match up on the original vertices of the triangle, use the rules above to transform them into one Sperner labeling; a fully-labeled triangle, which exists
by Sperner's lemma, now exhibits all three labels across the two Sperner labelings and both Sperner labelings exhibit at least two labels.
Higher-dimensional generalizations of Sperner's lemma to multiple labelings were conjectured by the third author and proven by Babson~\cite{babson2012}.
We will treat these and other extensions with new and simple proofs in Section~\ref{sec:generalizations}.

\begin{proof}[Proof 2:]	
	This second proof of Theorem~\ref{thm:secret} for $n=3$ will easily generalize to the case of $n$ roommates and still yield an algorithm to find the fair division of rent. It is based on the same piecewise linear map used in the second proof of Sperner's lemma.
	
	As in the first proof, construct a triangulated standard simplex in $\R^3$, and at each vertex, survey Larry and Moe about their room preferences. Instead of recording this as a tuple, construct two piecewise linear maps $\lambda_1, \lambda_2\colon \Delta \longrightarrow \Delta$ which reflect the preferences of Larry and Moe, respectively. That is, for each vertex $v$ of the triangulation, $\lambda_1$ maps $v$ to the vertex of the original triangle with the same label as Larry's preferred room at the price division given at vertex~$v$. The map $\lambda_2$ is defined in the same way using Moe's preferences. Then $\lambda_i$ is defined within each smaller triangle as the linear extension of the values at its vertices.
	
	Let $\lambda = \frac12(\lambda_1+\lambda_2)$ denote their average, which again is a piecewise linear map $\lambda\colon \Delta \longrightarrow \Delta$.
	The map $\lambda$ maps vertices of the subdivision of~$\Delta$ either to one of the three original vertices of~$\Delta$ (if the vertex receives
	the same label by both Larry and Moe) or to one of the three midpoints of edges (if the labelings do not agree on the vertex).
	
	Observe that for every subdivision vertex $v$ the vector $2\lambda(v)$ counts how often each label is exhibited in~$v$, that is,
	if $2\lambda(v) = (1,0,1)$ then $\lambda_1(v) = e_1$ and $\lambda_2(v) = e_3$ or vice versa.
	We can suppose that on each of the original vertices of the triangle Larry and Moe decide for the same room, and that they 
	decide for each room precisely once on one of the three original vertices.
	Then as before we check that $\lambda$ has degree one on the boundary of~$\Delta$ and thus there is a point $x \in \Delta$ with
	$\lambda(x) = (\frac13, \frac13, \frac13)$. The point $x$ lies in some small triangle $\tau$. 
	
	We claim that (1) $\tau$ exhibits all three labels, and (2) both $\lambda_1$ and $\lambda_2$ exhibit at least two labels. To see claim (1), note that if $\tau$ exhibited only two of the labels, then one of the coordinates of $\lambda(x)$ would be zero. To see claim (2) assume for contradiction that either $\lambda_1$ or $\lambda_2$ only exhibits one of the labels. Assume that label is 1. Then the first coordinate of $\lambda(v)$ will be greater than or equal to $\frac{1}{2}$ for each vertex $v$ of $\tau$. And so the first coordinate of $\lambda(y)$ for any $y\in \tau$ must be at least $\frac{1}{2}$, which contradicts $\lambda(x) = (\frac13,\frac13,\frac13)$. After the secret roommate selects their room, two remain. Since $\tau$ exhibits all three labels, each of the remaining labels will be exhibited by $\lambda_1$ or $\lambda_2$. Moreover, since $\lambda_1$ and $\lambda_2$ each exhibit at least two distinct labels, it is impossible that one person prefers neither room. Therefore, the remaining two rooms can be assigned in an envy-free way.
\end{proof}

\section{The general case}
\label{sec:general}

Generalizing the last section, here we give a proof of the main result, Theorem~\ref{thm:secret}: One can always find a fair division of the rent for an $n$-bedroom apartment given the subjective preferences of only $n-1$ roommates. The proof is a generalization of the second proof given in Section~\ref{sec:proof3case}.

\begin{proof}[Proof:]
	For $n$ roommates, we consider the standard $(n-1)$-simplex in $\R^n$. Its vertices lie on $e_1,...,e_n$, the standard basis of~$\R^n$, and $x_1+x_2+\cdots+x_n = 1$ for any points $(x_1,\dots,x_n)$ in the simplex. Similar to the $n=3$ case, each point in the simplex is a distribution of the rent and the fraction of the rent corresponding to the $i$th room is given by~$x_i$. Triangulate the simplex finely enough so that rent division in the same subdivision simplex are within a one-cent error-margin.
		
	For each of the $n-1$ given subjective preferences, we define a map $\lambda_j\colon \Delta_{n-1} \longrightarrow \Delta_{n-1}$ from the triangulated $(n-1)$-simplex to itself, defined as follows. For each vertex $v$ of the triangulated simplex, $\lambda_j(v)$ maps to a vertex of the original simplex, recording the $j$th roommate's preference. For example, if $\lambda_j(v) = e_3$ then person~$j$ prefers room 3 at the price distribution given at~$v$. After $\lambda_j$ is specified on each vertex of the triangulation, define $\lambda_j$ within each smaller simplex as the linear extension of its values on the vertices of the smaller simplex.
	
	We use that no roommate strictly prefers a free room over another free room to argue 
that we can impose that the choices of each roommate define a Sperner labeling of the triangulation of~$\Delta_{n-1}$.
The effect of this is, that up to a permutation of the vertices, each map $\lambda_j$ satisfies $\lambda_j(\sigma) \subseteq \sigma$ for each
face $\sigma$ of~$\Delta_{n-1}$. 

The vertices $e_1, e_2, \dots, e_n$ of~$\Delta_{n-1}$ correspond to the rooms $1,2,\dots, n$, where room $i$ is the only
non-free room at~$e_i$. We ask each roommate to decide for room $\pi(i) \coloneqq i+1$ at~$e_i$ for $i < n$ and room $\pi(n) \coloneqq 1$ at $e_n$.
This labeling of the vertices extends to a Sperner labeling such that the label of any vertex on the boundary
of~$\Delta_{n-1}$ corresponds to a free room. To see this notice that the rooms that are free for rent divisions in some face
$\sigma$ of~$\Delta_{n-1}$ are precisely those rooms that correspond to vertices of~$\Delta_{n-1}$ not contained in~$\sigma$.
Thus for any proper face $\sigma$ a vertex of $\sigma$ must be labeled by a free room, since $\pi$ has only one orbit, 
and hence such a free room is a valid label for vertices inside~$\sigma$.

We have shown that we can assume that the choices of each roommate define a Sperner labeling.
Thus, up to renaming vertex $e_i$ in the domain of $\lambda_j$ as~$e_{\pi(i)}$, 
each map $\lambda_j$ satisfies $\lambda_j(\sigma) \subseteq \sigma$ for each face of~$\Delta_{n-1}$. 
More precisely, let $f \colon \Delta_{n-1} \longrightarrow \Delta_{n-1}$ be the affine map that is 
defined on vertices by $f(e_i) = e_{i+1}$ for $i < n$ and $f(e_n) = e_1$. Then $f^{-1}(\lambda_j(\sigma)) \subseteq \sigma$ for every face $\sigma$ of~$\Delta_{n-1}$.
We will from here on tacitly assume this renaming of the vertices so that each $\lambda_j$ satisfies $\lambda_j(\sigma) \subseteq \sigma$
for every face $\sigma$ of~$\Delta_{n-1}$.

	Let $\lambda\colon \Delta_{n-1} \longrightarrow \Delta_{n-1}$ denote the average, $\lambda := \frac{1}{n-1}(\lambda_1 + \lambda_2 + \cdots + \lambda_{n-1})$. 
	Since each $\lambda_j$ fixes the faces of $\Delta_{n-1}$ setwise, so does their average~$\lambda$. Thus, as before, $\lambda$ is surjective. Let $x \in \Delta_{n-1}$ such that $\lambda(x) = (\frac{1}{n}, \dots, \frac{1}{n})$. That is, $x$ is mapped to the barycenter. Let $\tau$ be the simplex containing~$x$.
	
	Fix $k \in \{1, \dots, n-1\}$. We claim that any $k$-subset of the labelings $\lambda_i$ will exhibit at least $k+1$ labels within~$\tau$. Assume for contradiction that $\Lambda$ is some $k$-subset of the $\lambda_i$ that only exhibits the $k$ labels $e_{j_1}, \dots, e_{j_k}$. Then $\lambda(\tau)$ will be shifted toward the vertices $e_{j_i}$ and thus will not contain the barycenter. More precisely, for any vertex $v$ of $\tau$ and any $i \in \Lambda$, 
	$$\langle \lambda_i(v), e_{j_1} + \dots + e_{j_k} \rangle = 1$$
	and hence
	$$\langle \lambda(v), e_{j_1} + \dots + e_{j_k} \rangle \ge \frac{k}{n-1}.$$
	Since this holds for all vertices $v$ of~$\tau$, it also holds for any point inside~$\tau$. But,
	$$\langle \lambda(x), e_{j_1} + \dots + e_{j_k} \rangle = \langle (\frac{1}{n}, \dots, \frac{1}{n}), e_{j_1} + \dots + e_{j_k} \rangle = \frac{k}{n},$$
	which is a contradiction.
	
	This means any subset of $k$ roommates prefers at least $k+1$ rooms. This implies there is a fair division independently of which room is picked by the secretive roommate given by the labels of~$\tau$. This is because no matter which room the secretive roommate picks, any subset of $k$ (nonsecretive) roommates has $k$ rooms to pick from; construct a bipartite graph with vertices corresponding to the $n-1$ roommates on the one hand and vertices corresponding to the $n-1$ untaken rooms on the other. We add an edge for a pair of roommate and room if the roommate prefers this particular room. A fair rent division now corresponds to a perfect matching, which exists by Hall's marriage theorem: in any bipartite graph with bipartite sets $A$ and $B$ there exists a matching that entirely covers $A$ if and only if for every subset $W \subseteq A$ its neighborhood $N(W)$ satisfies $|N(W)| \ge |W|$.
\end{proof}

\section{Algorithmic aspects}
\label{sec:algorithm}

Our proof in the previous section described the construction of a piecewise linear map $\lambda\colon \Delta_{n-1} \longrightarrow \Delta_{n-1}$
that fixes the faces of $\Delta_{n-1}$ setwise: $\lambda(\sigma) \subseteq \sigma$. Such a map is necessarily surjective. A fair division of rent corresponds to a 
simplex $\tau$ of the triangulation of $\Delta_{n-1}$ such that there is an $x \in \tau$ that $\lambda$ maps to the barycenter of~$\Delta_{n-1}$. 
Here we describe a simple algorithm how to find the face~$\tau$. This algorithm does not use the surjectivity of~$\lambda$, but only that
$\lambda(\sigma) \subseteq \sigma$ for each face $\sigma$ of~$\Delta_{n-1}$. Thus our algorithm also gives an elementary proof of the 
surjectivity of~$\lambda$.

It is instructive to first consider low-dimensional cases. The algorithm for $n=2$ roommates just traverses the interval
$\Delta_1$ until for some edge $e$ of (a triangulation of) $\Delta_1$ the image $\lambda(e)$ contains the barycenter,
which must exist by the Intermediate Value Theorem.

In general, our algorithm is a trap-door argument, which we first describe for the case of a triangle $\Delta_2 = \conv\{e_1,e_2,e_3\}$.
We are given a triangulation $T$ of $\Delta_2$ and a map $\lambda \colon \Delta_2 \longrightarrow \Delta_2$ that interpolates 
linearly on every face of~$T$. The map $\lambda$ fixes the vertices and edges of $\Delta_2$ setwise. 
We will construct a path that starts in the vertex $e_1$ of~$\Delta_2$ and ends in a triangle 
$\sigma$ of~$T$ such that $\lambda(\sigma)$ contains the barycenter $\frac13(e_1+e_2+e_3)$ of~$\Delta_2$. We will now
describe the rooms and doors for our trap-door argument, or equivalently the vertices and edges of a graph $G$ such that 
following paths in this graph will lead to a triangle mapped to the barycenter by~$\lambda$. To build this graph we assume that 
$\lambda$ is generic in the sense that no vertex of $T$ gets mapped to the segment connecting the barycenter of $[e_1,e_2]$ to 
the barycenter of~$\Delta_2$.
The vertices of~$G$, or the rooms, are: 
\begin{compactitem}
	\item the vertex $e_1$ of $\Delta_2$,
	\item any edge $e$ of $T$ that subdivides the edge $[e_1,e_2]$ of $\Delta_2$ and such that $\lambda(e) \cap [e_1, \frac12(e_1+e_2)] \ne \emptyset$;
	that is, the image of $e$ under $\lambda$ intersects the segment from vertex $e_1$ to the barycenter of $[e_1,e_2]$,
	\item any triangle $\sigma$ of $T$ such that $\lambda(\sigma)$ intersects the segment $[\frac12(e_1+e_2), \frac13(e_1+e_2+e_3)]$
	connecting the barycenter of the edge $[e_1,e_2]$ to the barycenter of~$\Delta_2$. 
\end{compactitem}

The edges of~$G$, or the doors, are:
\begin{compactitem}
	\item between $e_1$ and the vertex corresponding to the edge of $T$ that contains $e_1$ and subdivides~$[e_1,e_2]$,
	\item between any two vertices corresponding to boundary edges of $T$ that share a common vertex $v$ with $\lambda(v) \in [e_1, \frac12(e_1+e_2)]$,
	\item between any two vertices corresponding to triangles of $T$ that share a common edge $e$ with 
	$\lambda(e) \cap[\frac12(e_1+e_2), \frac13(e_1+e_2+e_3)]\ne \emptyset$,
	\item between a boundary edge $e$ of $T$ with $\frac12(e_1+e_2) \in \lambda(e)$ and the incident triangle $\sigma$ of~$T$, i.e., the unique triangle $\sigma$
	that contains $e$ as an edge.
\end{compactitem}

We claim that a connected component of the graph~$G$ is a path from $e_1$ to some triangle $\sigma$ with $\frac13(e_1+e_2+e_3) \in \lambda(\sigma)$.
Notice that $e_1$ is incident to one other vertex in $G$ (corresponding to the unique edge of $T$ that has $e_1$ as a vertex and
subdivides~$[e_1,e_2]$). If $e$ is a boundary edge of $T$ with $\lambda(e) \subseteq [e_1, \frac12(e_1+e_2)]$ then in $G$ it is connected to 
two other boundary faces of $T$ (the ones that lie to the left and right of $e$ on the edge~$[e_1,e_2]$). If $e$ is a boundary edge of $T$ with $\frac12(e_1+e_2) \in \lambda(e)$
then precisely one of its vertices gets mapped to $[e_1, \frac12(e_1+e_2)]$, so $e$ is connected to one other edge of $T$ in~$G$. The other
neighbor of $e$ in $G$ is the unique triangle $\sigma$ of $T$ that contains $e$ as an edge. Now since generically line segments intersect
(boundaries of) triangles in either one or two edges, the vertices of $G$ corresponding to triangles have precisely two neighbors unless
the segment of points in the triangle $\sigma$ that $\lambda$ maps to $[\frac12(e_1+e_2), \frac13(e_1+e_2+e_3)]$ intersect $\sigma$ in
only one edge. In that case $\frac13(e_1+e_2+e_3) \in \lambda(\sigma)$. Thus $G$ is a graph where all vertices 
have degree one or two, and have degree one if and only if they correspond to a triangle that gets mapped to the barycenter of~$\Delta_2$ or correspond to our starting point~$e_1$.
Starting in the vertex $e_1$ and following the edges of $G$ we must end up in such a triangle, as desired.

To summarize the algorithm, we start walking in the vertex $e_1$ and traverse along the edge $[e_1,e_2]$ until we hit an edge of $T$ that is
mapped to the barycenter of $[e_1,e_2]$. From there we walk inwards into the triangle $\Delta_2$ following a path of triangles whose image
under $\lambda$ intersects $[\frac12(e_1+e_2), \frac13(e_1+e_2+e_3)]$. This either ends in a triangle of $T$ that gets mapped to the barycenter
$\frac13(e_1+e_2+e_3)$, or we return to edges subdividing $[e_1,e_2]$. However, in the latter case we must leave the edge $[e_1,e_2]$
and follow a path of triangles again. After finitely many trips back to the edge $[e_1,e_2]$ we must end up in a triangle mapped to the barycenter.

This construction and algorithm easily generalize to higher dimensions. We work with the faces 
$$e_1, \conv\{e_1,e_2\}, \conv\{e_1,e_2,e_3\}, \dots, \conv\{e_1, \dots, e_n\} = \Delta_{n-1}$$
and their barycenters $b_k = \sum_{i=1}^k \frac1ke_i$. All faces $\sigma$ of the triangulation $T$ of $\Delta_{n-1}$ that subdivide one of these
faces, say $\conv\{e_1, \dots, e_k\}$, and such that $\lambda(\sigma)$ intersects the segment $[b_{k-1},b_k]$ that joins the barycenter of
$\conv\{e_1, \dots, e_{k-1}\}$ to that of $\conv\{e_1, \dots, e_k\}$ make up the vertices of~$G$. The vertex $e_1$ of $\Delta_n$ is a vertex of 
$G$ as well. Two such faces $\sigma_1$ and $\sigma_2$ of dimension $k$ are connected by an edge in $G$ if they share a common
$(k-1)$-face $\tau$ such that $\lambda(\tau)$ intersects $[b_k, b_{k+1}]$. We assume that if $\lambda(\tau)$ intersects $[b_k, b_{k+1}]$, 
then there is a point $x$ in the relative interior of $\tau$ such that $\lambda(x) \in [b_k, b_{k+1}]$. This can be achieved by slightly perturbing the 
barycenters~$b_k$.
Moreover, there is an edge between $k$-face $\sigma$ and $(k-1)$-face
$\tau$ if $\tau$ is a face of $\sigma$ in $T$ and $b_k \in \lambda(\tau)$.

A line segment generically cannot intersect a $k$-face in more than two of its $(k-1)$-faces and it intersects in precisely one $(k-1)$-face if
it ends inside the $k$-face. Thus our reasoning for $\Delta_2$ also applies to this higher-dimensional construction and starting in the vertex
$e_1$ of $G$ we can follow edges of $G$ to end up in an $(n-1)$-face $\sigma$ of $T$ with $\frac1n\sum_{i=1}^n e_i \in \lambda(\sigma)$.

\section{Generalizations of Sperner's Lemma}
\label{sec:generalizations}

The methods of Section~\ref{sec:general} actually yield generalizations of Sperner's lemma to multiple labelings.
Fix a triangulation of the $n$-simplex and several Sperner labelings $\lambda_1, \dots, \lambda_m$ of it. 
We will always assume that these labelings match up on the original $n+1$ vertices of~$\Delta_n$. By Sperner's lemma each of these
labelings has a fully-labeled simplex. It is simple to come up with examples where no pair of these respective fully-labeled simplices coincide. 
In attempting to understand how many labels a single simplex must exhibit across the $m$ Sperner labelings, there are two natural
questions:
\begin{compactenum}[1.]
	\item \label{item:primal}
	How can we constrain $(n+1)$-tuples $(k_0, \dots, k_n)$ of nonnegative integers such that there is a simplex $\tau$ that exhibits
	the $i$th label $k_i$ times across the $m$ Sperner labelings $\lambda_1, \dots, \lambda_m$? 
	\item \label{item:dual}
	Dually, how can we constrain $m$-tuples $(k_1, \dots, k_m)$ of nonnegative integers such that there is a simplex $\tau$ on which
	$\lambda_i$ exhibits $k_i$ pairwise distinct labels?
\end{compactenum}

We will relate the first question to convex hulls of \emph{lattice points}, i.e., points with integer coordinates, 
in $m \cdot \Delta_n=\{x \in \R^{n+1} \: | \: \sum x_i = m, x_i \ge 0\}$, the $n$-simplex scaled by~$m$.
We will show that the label multiplicities $(k_0, \dots, k_n)$ that must occur are given by sets of $n+1$ 
lattice points in $m \cdot \Delta_n$ whose convex hulls all intersect in a common point $y$ that are maximal 
with this property, that is, any other convex hull of $n+1$ lattice points will not contain~$y$.

For example, given two Sperner labelings $\lambda_1$ and $\lambda_2$ of a triangulation of the triangle $\Delta_2$,
we need to understand intersections of convex hulls of three lattice points in $2\cdot \Delta_2$. The relevant lattice
points are the vertices $(2,0,0)$, $(0,2,0)$, $(0,0,2)$ and midpoints of edges $(0,1,1)$, $(1,0,1)$, $(1,1,0)$.
Say we consider the convex hulls of lattice points that capture $y = (2-3\varepsilon, 2\varepsilon, \varepsilon)$ for
some small $\varepsilon > 0$. The point $y$ is close to the vertex $(2,0,0)$ and even closer to the edge between
$(2,0,0)$ and $(0,2,0)$ without being on it. All possible choices of three lattice points whose convex hulls contain $y$
have the lattice point $(2,0,0)$, another lattice point on the edge between
$(2,0,0)$ and $(0,2,0)$, i.e., one of $(0,2,0)$ or $(1,1,0)$, and one lattice point that is not on that edge, i.e., one of $(1,0,1)$, $(0,0,2)$, or $(0,1,1)$.
So there are exactly six choices of three lattice points whose convex hulls capture~$y$. This geometric fact translates
into the following combinatorial fact about Sperner labelings $\lambda_1$ and $\lambda_2$ of a triangle: there always
is a smaller triangle with vertices $v_1, v_2, v_3$ such that
\begin{compactitem}
	\item $v_1$ exhibits the first label twice ($\lambda_1(v_1) = \lambda(v_1) = 1$) corresponding to the lattice point $(2,0,0)$,
	\item $v_2$ either exhibits the second label twice or the first and second label once corresponding to $(0,2,0)$ or $(1,1,0)$,
	\item and $v_3$ must exhibit the third label at least once corresponding to $(1,0,1)$, $(0,0,2)$, or $(0,1,1)$.
\end{compactitem}

A priori, the simplex that is labeled by $\lambda_1, \dots, \lambda_m$ and the scaled simplex that encodes all possible $(n+1)$-tuples
$(k_0, \dots, k_n)$ of label multiplicities are entirely different objects. While we should expect no direct relation between these two simplices, 
it turns out that we get constraints on the $k_i$ by thinking of the simplices as the same geometric object.

\begin{thm}
Let $\lambda_1, \dots, \lambda_m$ be Sperner labelings of a triangulation $T$ of~$\Delta_n$. Let $y \in m \cdot \Delta_n$ be some point that is not in the convex hull of any $n$ lattice points in~$m\cdot \Delta_n$. Then there is a facet $\sigma$ of $T$ and an ordering of its vertices $v_1, \dots, v_{n+1}$ such that the point $y$ is contained in $\conv\{y_1, \dots, y_{n+1}\}$, where 
$y_i \in m \cdot \Delta_n$ denotes the lattice point whose $j$th coordinate is the number of times $j$ labels $v_i$.
\end{thm}

\begin{proof}[Proof:]
	Let $\lambda = \lambda_1+ \dots + \lambda_m \colon \Delta_n \longrightarrow m\cdot \Delta_n$. As before since the $\lambda_i$ are
	Sperner labelings the map $\lambda|_{\partial\Delta_n}$ has degree one as a map to the boundary of $m\cdot \Delta_n$; the map
	$\lambda$ satisfies $\lambda(\sigma) \subseteq m \cdot \sigma$ for any face $\sigma$ of~$\Delta_n$. Just as before the average 
	$\frac1m \lambda$ fixes faces setwise. Thus there
	is an $x \in \Delta_n$ with $\lambda(x) = y$. Let $\tau$ be a face of the triangulation of $\Delta_n$ that contains~$x$. The map~$\lambda$
	maps vertices of the triangulation of~$\Delta_n$ to lattice points of~$m \cdot \Delta_n$. Since $y$ is not in the convex hull of fewer than
	$n+1$ lattice points in $m \cdot \Delta_n$, the vertices of~$\tau$ must be mapped precisely to the elements of a set of $n+1$ lattice points
	in $m\cdot \Delta_n$ whose convex hull captures~$y$.
\end{proof}

Question~\ref{item:dual} can be approached in much the same way. Instead of defining the map $\lambda$ as the sum or average of the 
piecewise linear extensions of the Sperner labelings as before, we now take a biased average with weights according to how many labels
each Sperner labeling is supposed to exhibit. The third author conjectured in his dissertation that $\sum k_i = n+m$ is a valid constraint for
question~\ref{item:dual}. This was recently proven by Babson~\cite{babson2012}. We give a different proof below in the spirit of 
the other proofs of this manuscript.

\begin{thm}
Let $\lambda_1,\ldots,\lambda_m$ be $m$ Sperner labelings of a triangulation of~$\Delta_n$ and let $k_1,\ldots,k_m$ be $m$ positive integers 
summing up to~$n+m$. Then there exists a simplex $\tau$ such that $\lambda_j$ exhibits at least $k_j$ pairwise distinct labels on~$\tau$ for all~$j$.
\end{thm}

\begin{proof}[Proof:]
	Let $\alpha_j = \frac{1}{n+1}(k_j+\frac1m-1)$ for $1 \le j \le m$. Then since $\sum_j k_j =n+m$ we have that $\sum_j \alpha_j = 1$. 
	Thus $\lambda = \sum_j \alpha_j\lambda_j$ is a map $\Delta_n \longrightarrow \Delta_n$, and $\lambda$ satisfies $\lambda(\sigma) \subseteq \sigma$
	for each face $\sigma$ of~$\Delta_n$ as usual. Let $x \in \Delta_n$ with $\lambda(x) = (\frac{1}{n+1}, \dots, \frac{1}{n+1})$ and let $\tau$ be a facet of the 
	triangulation of $\Delta_n$ containing~$x$. Denote the vertices of $\tau$ by $v_0, \dots, v_n$ and let $x = \sum_i \mu_iv_i$ for nonnegative
	$\mu_i$ with $\sum_i \mu_i = 1$. Define for $i = 1, \dots, n+1$ and $j = 1, \dots, m$
	$$\beta_{ij} = \alpha_j\cdot\sum_{\{k | \lambda_j(v_k) = e_i\}} \mu_k.$$
	
	Since $\sum_i \mu_i = 1$ we have that $\sum_i \beta_{ij} = \alpha_j$ for every~$j$. The choice of $x$, definition of~$\lambda$, and piecewise linearity
	of the $\lambda_j$ imply that
	$$\left(\frac{1}{n+1}, \dots, \frac{1}{n+1}\right) = \lambda(x) = \sum_j \alpha_j\lambda_j(x) = \sum_j \alpha_j \sum_{k=0}^n \mu_k\lambda_j(v_k)$$
	and thus $\sum_j \beta_{ij} = \frac{1}{n+1}$. Since in particular $0 \le \beta_{ij} \le \frac{1}{n+1}$ and we already know that $\sum_i \beta_{ij} = \alpha_j$, 
	for each~$j$ the number of indices~$i$ such
	that $\beta_{ij} > 0$ is at least $\alpha_j(n+1) > k_j-1$. Now $\beta_{ij} > 0$ implies that there is a vertex $v$ of $\tau$ with $\lambda_j(v) = e_i$, 
	and thus $\tau$ receives at least $k_j$ distinct labels by~$\lambda_j$.
\end{proof}

\section*{Acknowledgments}
\label{sec:ack}
	Figure~\ref{spernerlabeling} and Figure~\ref{labelwithdoors} are reproduced with permission from~\cite{houston-edwards2017} 
	and copyright 2017, PBS Infinite Series and Ray Lux.



\small
\bibliographystyle{amsplain}

\end{document}